\documentclass[11pt]{article}
\usepackage{t1enc}
\usepackage[latin1]{inputenc}
\usepackage[english]{babel}
\usepackage{amsmath,amsthm}
\usepackage{amsfonts}
\usepackage{latexsym}
\usepackage[dvips]{graphicx}
\usepackage{graphicx}
\usepackage{mathrsfs}
\title{On the inertia index of a mixed graph with the matching number}
\author{ Shengjie  He$^1$, Rong-Xia Hao$^1$\footnote{Corresponding author.
Emails: he1046436120@126.com (Shengjie  He), rxhao@bjtu.edu.cn (Rong-Xia Hao), yuaimeimath@163.com (Aimei Yu)}, Aimei Yu$^1$\\
{\small\em 1. Department of Mathematics, Beijing Jiaotong University, Beijing,
100044, China}\\
  }

\date{} \textwidth 16cm \textheight 22cm \topmargin 0 cm \hoffset
-1.5 cm \voffset 0cm

\newtheorem{theorem}{Theorem}[section]
\newtheorem{lemma}[theorem]{Lemma}

\newtheorem{corollary}[theorem]{Corollary}

\begin{document}
\baselineskip 0.50cm \maketitle

\begin{abstract}
A mixed graph $\widetilde{G}$ is obtained by orienting some edges of $G$, where
$G$ is the underlying graph of $\widetilde{G}$. The positive inertia index, denoted by $p^{+}(G)$,  and the negative inertia index, denoted by $n^{-}(G)$, of a mixed
graph $\widetilde{G}$ are the integers specifying the numbers of positive and negative eigenvalues
of the Hermitian adjacent matrix of $\widetilde{G}$, respectively. In this paper, we study the positive and negative inertia index of the mixed unicyclic graph. Moreover, we give the upper and lower bounds of the positive and negative inertia index of the mixed graph, and characterize the mixed graphs which attain the upper and lower bounds respectively.


\end{abstract}

\section{Introduction}
In this paper, we consider only graphs without multiedges and  loops. A $undirected$ $graph$ $G$ is denoted by $G=(V, E)$, where $V$ is the vertex set and $E$ is the edge set.
A $mixed$ $graph$ $\widetilde{G}$ is obtained by orienting some edges of $G$, where
$G$ is the underlying graph of $\widetilde{G}$.
Denote by $\widetilde{P_n}$, $\widetilde{S_n}$ and $\widetilde{C_n}$ a mixed path, mixed star and mixed cycle on $n$ vertices, respectively.
We refer to \cite{West} for terminologies and notations undefined here.

A vertex of a mixed graph $\widetilde{G}$ is called a $pendant$ $vertex$ if it is a vertex of degree one in $\widetilde{G}$, whereas a vertex of $\widetilde{G}$ is called a $quasi$-$pendant$ vertex if it is adjacent to a vertex of degree one in $\widetilde{G}$ unless it is a pendant vertex.
An $induced$ $subgraph$ $\widetilde{H}$ of $\widetilde{G}$ is a mixed graph such that the underlying graph of  $\widetilde{H}$ is an induced subgraph of the underlying graph of $\widetilde{G}$ and each edge of $\widetilde{H}$ has the same orientation (or non-orientation) as that in $\widetilde{G}$. For
$X \subseteq V(\widetilde{G})$, $\widetilde{G}-X$ is the mixed subgraph obtained from $\widetilde{G}$ by deleting all vertices in $X$ and
all incident edges or arcs. In particular, $\widetilde{G}-\{ x \}$ is usually written as $\widetilde{G}-x$ for simplicity.
For the sake of clarity, we use the notation $\widetilde{G}-\widetilde{H}$ instead of $\widetilde{G}-V(\widetilde{H})$ if $\widetilde{H}$ is an induced
subgraph of $\widetilde{G}$.

For an undirected $G$, denote by $c(G)$ the $dimension$ of cycle space of $G$,
that is $c(G)=|E(G)|-|V(G)|+\omega(G)$, where $\omega(G)$ is the number of connected components of $G$.
Two distinct edges in a graph $G$ are $independent$ if they do not have common end-vertex in $G$. A set of pairwise independent edges of $G$ is called a $matching$, while a matching with the maximum
cardinality is a $maximum$ $matching$ of $G$. The $matching$ $number$ of $G$, denoted by $m(G)$,
is the cardinality of a maximum matching of $G$.
For a mixed graph $\widetilde{G}$, the dimension of cycle space, denoted by $c(\widetilde{G})$,  and matching number, denoted by $m(\widetilde{G})$, are defined to be the dimension of cycle space and
matching number of its underlying graph, respectively.

The $Hermitian$-$adjacency$ $matrix$ of a mixed graph $\widetilde{G}$ of order $n$ is the $n \times n$ matrix $H(\widetilde{G}) = (h_{kl})$, where $h_{kl}=-h_{lk}=\mathbf{i}$ if
there is a directed edge (or an arc) from $v_{k}$ to $v_{l}$, where $\mathbf{i}$ is the imaginary number unit and $h_{kl}=h_{lk}=1$ if $v_{k}$ is connected to $v_{l}$ by an undirected edge, and $h_{kl}=0$ otherwise.
It is easy to see that $H(\widetilde{G})$ is a Hermitian  matrix, i.e., its conjugation and transposition is itself, that  is $H=H^{\ast}:=\overline{H}^{T}$. Thus
all its eigenvalues are real. The $positive$ $inertia$ $index$ (respectively, the $negative$ $inertia$ $index$) of a mixed graph $\widetilde{G}$, denoted by $p^{+}(\widetilde{G})$ (resp. $n^{-}(\widetilde{G})$), is defined to be the number of positive eigenvalues (resp. negative eigenvalues)
of $H(\widetilde{G})$. The $rank$ of a mixed graph $\widetilde{G}$, denoted by $rank(\widetilde{G})$,
is exactly the sum of $p^{+}(\widetilde{G})$ and $n^{-}(\widetilde{G})$. The $nullity$ of a mixed graph $\widetilde{G}$, denoted by $\eta(\widetilde{G})$, the algebraic multiplicity of the zero eigenvalues of $H(\widetilde{G})$. It is obviously that $\eta(\widetilde{G})=n-p^{+}(\widetilde{G})-n^{-}(\widetilde{G})$, where $n$ is the order of $\widetilde{G}$. The inertia index and nullity of graph have been studied by
many researchers, we refer to \cite{WONG,MAH,MAH2,RULA}.

The value of a mixed walk $\widetilde{W}=v_{1}v_{2}v_{3} \cdots v_{l}$ is $h(\widetilde{W})=h_{12}h_{23} \cdots h_{(l-1)l}$. A mixed
walk is positive or negative if $h(\widetilde{W}) = 1$ or $h(\widetilde{W}) = -1$, respectively. Note that for one
direction the value of a mixed walk or a mixed cycle is $\alpha$, then for the reversed direction
its value is $\overline{\alpha}$. Thus, if the value of a mixed cycle is 1 (resp. $-$1) in a direction, then its
value is 1 (resp. $-$1) for the reversed direction. In these situations, we just termed this
mixed cycle as a positive (resp. negative) mixed cycle without mentioning any direction.
A mixed graph is $positive$ (resp. $negative$) if each its mixed cycle is positive (resp. negative).
An $elementary$ graph is a mixed graph such that every component is an edge or a
mixed cycle, and every its edge-component is defined to be positive. A $real$
$elementary$ $subgraph$ of a mixed graph $G$ is an elementary subgraph such that all its mixed cycles are positive or negative.
For a mixed cycle $\widetilde{C}$ of a real elementary subgraph $\widetilde{G'}$, the $signature$ of $\widetilde{C}$, denoted by $\sigma(\widetilde{C})$, is defined as $|f-b|$, where $f$ denotes the number of forward-oriented edges and $b$ denotes the number of backward-oriented edges of $\widetilde{C}$.
The signature of a real elementary subgraph $\widetilde{G'}$, denoted by $\sigma(\widetilde{G'})$, is defined as the sum of the signatures of all the cycles in $\widetilde{G'}$.
Obviously, if $\widetilde{G'}$ is a real elementary subgraph of a mixed graph $\widetilde{G}$, then $\sigma(\widetilde{C'})$ is even for each cycle $\widetilde{C'}$ in $\widetilde{G'}$.

In recent years, the study on the Hermitian adjacent matrix and the characteristic polynomial of mixed graphs received increased attention. In \cite{LXL}, Liu and Li investigated the properties of the coefficients of the characteristic polynomial of
mixed graphs and cospectral problems among mixed graphs. Mohar \cite{MOHAR2} characterized all the mixed graphs with rank equal to 2. Wang et al. \cite{YBJ} studied the graphs with $H$-rank 3. Chen et al.
researched the bounds relationship between the rank and the matching number of a mixed graph $\widetilde{G}$ and the mixed graphs which attained the upper and lower bounds are characterized, respectively in \cite{LSC}. Daugherty \cite{SDAU} studied the inertia index of undirected unicyclic graphs and the implications
for closed-shells.  In \cite{FANYZ}, Fan and Wang studied bounds of the positive inertia index and negative inertia index of the adjacent matrix of an undirected graph.

In this paper, we study the positive and negative inertia index of the mixed unicyclic graph. Moreover, we give the upper and lower bounds of the positive and negative inertia index of the mixed graph, and characterize the mixed graphs which attain the upper and lower bounds respectively.

\section{Preliminaries}

In this section, we show some lemmas which will be useful in the following section.

\begin{lemma}\label{L21} $\rm{(Descartes' \ sign \ rule)}$  The number of positive roots of a polynomial $f(x)=f_{0}x^{n}+f_{1}x^{n-1}+\cdots +f_{n}$ with all real roots is equal to
the number of sign changes of $f_{i}$ proceeding from $f_{0}$ to $f_{n}$, ignoring $f_{i}=0$.
\end{lemma}

Suppose that $\lambda_{1} \geq \lambda_{2} \geq \cdots \geq  \lambda_{n}$ and $\kappa_{1} \geq \kappa_{2} \geq \cdots \geq  \kappa_{n-t}$
(where $t \geq 1$ is an integer) be two sequences of  real numbers. We say that the sequences $\lambda_{l}$ $(1 \leq l \leq n)$ and $\kappa_{j}$ $(1 \leq j \leq n-t)$
interlace if for every $s=1, 2, \cdots, n-t$, we have
$$\lambda_{s} \geq  \kappa_{s} \geq \lambda_{s+t}.$$

The usual version of the eigenvalue interlacing property states that the eigenvalues of any principal submatrix of a Hermitian matrix
interlace those of the whole matrix.

\begin{lemma} \label{L22}{\rm\cite{MOHAR}} If $H$ is a Hermitian matrix and $B$ is a principal submatrix of $H$, then the eigenvalues
of $B$ interlace those of $H$.
\end{lemma}

The characteristic polynomial of the Hermitian-adjacent matrix of a mixed graph $\widetilde{G}$ is defined as
$$\phi(\widetilde{G}, \lambda)=det(\lambda I_{n}-H(\widetilde{G}))=\lambda^{n}+a_{1}\lambda^{n-1}+a_{2}\lambda^{n-2}+ \cdots +a_{n},$$
where $I_{n}$ denoted the unit matrix of  order $n$.

\begin{lemma} \label{L23}{\rm\cite{LXL}} Let $\mathscr{B}_{j}$ be the set of real elementary subgraph with $j$ vertices of $\widetilde{G}$.
Then the coefficient $a_{j}$ of the characteristic polynomial of the Hermitian-adjacent matrix of  $\widetilde{G}$ is
$$a_{j}=\sum\limits_{B \in \mathscr{B}_{j}}(-1)^{\frac{1}{2}\sigma(B)+\omega(B)} \cdot 2^{ c(B)}, j=1, 2, \cdots, n,$$
where $\omega(B)$ denotes the number of components of $B$ and $c(B)$ is the number of cycles in $B$.
\end{lemma}

The graph $\widetilde{G} \in \mathscr{G}$ if $\widetilde{G}$ has the following properties: (1) $\widetilde{G}$ contains at least one cycle but is not the disjoint union the disjoint cycles, (2) any two cycles of $\widetilde{G}$ share no common vertices if $\widetilde{G}$ contains more than one cycle. Contracting each cycle of a graph $\widetilde{G} \in \mathscr{G}$ into a vertex (called cyclic vertex), we obtain a forest denoted by $T_{\widetilde{G}}$. Denote by $[T_{\widetilde{G}}]$ the subgraph of
$T_{G}$ induced by all non-cyclic vertices. For a graph $\widetilde{G} \in \mathscr{G}$, let $\mathscr{F}(\widetilde{G})$ be the set of edges of $\widetilde{G}$ which has an endpoint on a cycle and the other
endpoint outside the cycle.

\begin{lemma} \label{L24}{\rm\cite{SDAU}} Let $G$ be a undirected unicyclic graph
containing the cycle $C_{q}$. Then

$$(p^{+}(G), n^{-}(G))=\left\{
             \begin{array}{ll}
               (m(G)-1, m(G)-1), & \hbox{if $q=4k$ and $M \cap \mathscr{F}(G)=\emptyset$ for any }\\
                                  & \hbox{maximum matching $M$ of $G$;} \\
               (m(G)+1, m(G)), & \hbox{if $q=4k+1$ and $m(G)=m(G-C_{q})+\frac{q-1}{2}$;} \\
               (m(G), m(G)+1), & \hbox{if $q=4k+3$ and $m(G)=m(G-C_{q})+\frac{q-1}{2}$;} \\
               (m(G), m(G)), & \hbox{otherwise.}
             \end{array}
           \right.
$$
\end{lemma}

\begin{lemma} \label{L25}{\rm\cite{YBJ}} Let $\widetilde{C_{n}}$ be a mixed cycle with $n$ vertices. Then
$$rank(\widetilde{C_{n}})=\left\{
             \begin{array}{ll}
               n-1, & \hbox{if $n$ is odd, $\sigma(\widetilde{C_{n}})$ is odd;} \\
               n, & \hbox{if $n$ is odd, $\sigma(\widetilde{C_{n}})$ is even;} \\
               n, & \hbox{if $n$ is even, $\sigma(\widetilde{C_{n}})$ is odd;} \\
               n, & \hbox{if $n$ is even, $n+\sigma(\widetilde{C_{n}}) \equiv 2 \ (\rm{mod} \ 4)$;} \\
               n-2, & \hbox{if $n$ is even, $n+\sigma(\widetilde{C_{n}}) \equiv 0 \ (\rm{mod} \ 4)$.}
             \end{array}
           \right.
$$
\end{lemma}

\begin{lemma} \label{TREE}{\rm\cite{DCVE}} Let $\widetilde{G}$ be a mixed tree. Then
$p^{+}(\widetilde{G})=n^{-}(\widetilde{G})=m(\widetilde{G})$.
\end{lemma}

\begin{lemma} \label{COMPON}
Let $\widetilde{G}$ be a mixed graph. If $\widetilde{G_{1}}, \widetilde{G_{2}}, \cdots, \widetilde{G_{k}}$
are connected components of $\widetilde{G}$. Then $p^{+}(\widetilde{G})=\sum_{i=1}^{k}p^{+}(\widetilde{G_{i}})$ and $n^{-}(\widetilde{G})=\sum_{i=1}^{k}n^{-}(\widetilde{G_{i}})$.
\end{lemma}
\begin{proof}
It can be checked that $\phi(\widetilde{G}, \lambda)=\prod_{i=1}^{k}\phi(\widetilde{G_{i}}, \lambda)$. Thus, we have the results.
\end{proof}

\begin{lemma}\label{DV}
Let $\widetilde{G}$ be a  mixed graph with a vertex $u$. Then
$$p^{+}(\widetilde{G})-1 \leq p^{+}(\widetilde{G}-u) \leq p^{+}(\widetilde{G}), n^{-}(\widetilde{G})-1 \leq n^{-}(\widetilde{G}-u) \leq n^{-}(\widetilde{G}). $$
Moreover,
$$rank(\widetilde{G})-2 \leq rank(\widetilde{G}-u) \leq rank(\widetilde{G}).$$
\end{lemma}
\begin{proof}
By Lemma \ref{L22}, the eigenvalues of the Hermitian-adjacent matrix of $\widetilde{G}$ interlace those of $\widetilde{G}-u$. The result follows.
\end{proof}

\begin{lemma}\label{L29} Let $\widetilde{G}$ be a mixed graph with a pendant vertex $u$ and a quasi-pendant vertex $v$ which is
adjacent to $u$. Let $\widetilde{G'}=\widetilde{G}-\{u, v\}$, then
$$p^{+}(\widetilde{G'})=p^{+}(\widetilde{G})-1, n^{-}(\widetilde{G'})=n^{-}(\widetilde{G})-1.$$
\end{lemma}
\begin{proof} Let $H(\widetilde{G})$ and $H(\widetilde{G'})$ be the Hermitian adjacency matrix of $\widetilde{G}$ and $\widetilde{G'}$, respectively.
Then, $H(\widetilde{G})$ can be written as
$$\left(
  \begin{array}{ccc}
    0 & \alpha & \mathbf{0} \\
    \alpha^{*} & 0 & \boldsymbol{\beta} \\
    \mathbf{0}^{T} & \boldsymbol{\beta}^{T} & H(\widetilde{G'}) \\
  \end{array}
\right).
$$
where $\alpha \in \{0, 1, \mathbf{\rm{i}}, \mathbf{\rm{-i}}  \}$. It can be checked that $H(\widetilde{G})$ is congruent to the following type matrix:
$$\left(
  \begin{array}{ccc}
    0 & \alpha & \mathbf{0} \\
    \alpha^{*} & 0 & \mathbf{0} \\
    \mathbf{0}^{T} & \mathbf{0}^{T} & H(\widetilde{G'}) \\
  \end{array}
\right).
$$

The result follows.
\end{proof}

\section{The positive and negative inertia index of the mixed graphs}

In this section, we study the positive and negative inertia index of the mixed cycle with $n$ vertices
firstly in Theorem \ref{T31}. In what follows, we present the positive and negative inertia index of the mixed unicyclic graph in Theorem \ref{UNI} by the results of Lemmas \ref{L32}-\ref{L34}.

\begin{theorem}\label{T31}  Let $\widetilde{C_{n}}$ be a mixed cycle with $n$ vertices. Then
$$  (p^{+}(\widetilde{C_{n}}), n^{-}(\widetilde{C_{n}})) = \left\{
  \begin{array}{ll}
      (\frac{n}{2}, \frac{n}{2}), & \hbox{ if $n$ is even, $\sigma(\widetilde{C_{n}})$ is odd;} \\
      (\frac{n}{2}, \frac{n}{2}), & \hbox{ if $n$ is even, $\sigma(\widetilde{C_{n}})$ is even and $n+\sigma(\widetilde{C_{n}}) \equiv 2 \ (\rm{mod} \ 4)$;} \\
      (\frac{n-2}{2}, \frac{n-2}{2}), & \hbox{ if $n$ is even, $\sigma(\widetilde{C_{n}})$ is even and $n+\sigma(\widetilde{C_{n}}) \equiv 0 \ (\rm{mod} \ 4)$;} \\
      (\frac{n-1}{2}, \frac{n-1}{2}), & \hbox{ if $n$ is odd, $\sigma(\widetilde{C_{n}})$ is odd;} \\
      (\frac{n+1}{2}, \frac{n-1}{2}), & \hbox{ if $n \equiv 1 \ (\rm{mod} \ 4)$, $\sigma(\widetilde{C_{n}}) \equiv 0 \ (\rm{mod} \ 4)$;} \\
      (\frac{n-1}{2}, \frac{n+1}{2}), & \hbox{ if $n \equiv 3 \ (\rm{mod} \ 4)$, $\sigma(\widetilde{C_{n}}) \equiv 0 \ (\rm{mod} \ 4)$;} \\
      (\frac{n-1}{2}, \frac{n+1}{2}), & \hbox{ if $n \equiv 1 \ (\rm{mod} \ 4)$, $\sigma(\widetilde{C_{n}}) \equiv 2 \ (\rm{mod} \ 4)$;} \\
      (\frac{n+1}{2}, \frac{n-1}{2}), & \hbox{ if $n \equiv 3 \ (\rm{mod} \ 4)$, $\sigma(\widetilde{C_{n}}) \equiv 2 \ (\rm{mod} \ 4)$.}
    \end{array}
  \right.
$$
\end{theorem}
\begin{proof}
Let $\phi(\widetilde{C_{n}}, \lambda)=\sum_{i=0}^{n}a_{i}\lambda^{n-i}$ be the characteristic polynomial of the Hermitian-adjacent matrix of $\widetilde{C_{n}}$.

If $n$ is even.  It is obviously that there is no real elementary subgraph with odd order. By Lemma~\ref{L23}, $a_{j}=0$ if $j$ is odd, then $p^{+}(\widetilde{C_{n}})= n^{-}(\widetilde{C_{n}})$. By Lemma~\ref{L25}, we have the desired results when $n$ is even.

If $n$ is odd and $\sigma(\widetilde{C_{n}})$ is odd. It can be checked that there is no real elementary subgraph with odd order. By Lemma~\ref{L23}, $a_{l} = 0$ when $l$ is odd and $a_{j} \neq 0$ when $j$ is even. Thus $\phi(\widetilde{C_{n}}, \lambda)=\lambda (\lambda^{n-1}+a_{2}\lambda^{n-3}+\cdots +a_{n-1})$ and $p^{+}(\widetilde{C_{n}})= n^{-}(\widetilde{C_{n}})=\frac{n-1}{2}$.

If $n$ is odd and $\sigma(\widetilde{C_{n}}) \equiv 0 \ (\rm{mod} \ 4)$.  It can be checked that
the characteristic polynomial of $H(\widetilde{C_{n}})$ is same as the characteristic polynomial of
the adjacent matrix of the undirected $n$-cycle. Moreover, $\widetilde{C_{n}}$ is the only real elementary subgraph with order $n$. By Lemma~\ref{L23}, $a_{j}=0$ for odd $j$ with $j < n$ and $a_{n}=-2$ and
$\phi(\widetilde{C_{n}}, \lambda)=\lambda(\lambda^{n-1}+a_{2}\lambda^{n-3}+\cdots +a_{n-1})-2$. Then by Lemmas~\ref{L24} and \ref{L25},  we have the results if
$n$ is odd and $\sigma(\widetilde{C_{n}}) \equiv 0 \ (\rm{mod} \ 4)$.

If $n$ is odd and $\sigma(\widetilde{C_{n}}) \equiv 2 \ (\rm{mod} \ 4)$. By similar discussion with the case that $n$ is odd and $\sigma(\widetilde{C_{n}}) \equiv 0 \ (\rm{mod} \ 4)$, it can be checked that $a_{n}=2$ and
$\phi(\widetilde{C_{n}}, \lambda)=\lambda(\lambda^{n-1}+a_{2}\lambda^{n-3}+\cdots +a_{n-1})+2$ (here the coefficient $a_{i}$ is same as the case that $n$ is odd and $\sigma(\widetilde{C_{n}}) \equiv 0 \ (\rm{mod} \ 4)$ for $i=0, 1, \cdots, n-1$). Then if $x$ is a characteristic eigenvalue of the case that $n$ is odd and $\sigma(\widetilde{C_{n}}) \equiv 0 \ (\rm{mod} \ 4)$, $-x$ is a characteristic eigenvalue of the case that $n$ is odd and $\sigma(\widetilde{C_{n}}) \equiv 2 \ (\rm{mod} \ 4)$. Thus, we have the desired results.

This completes the proof.
\end{proof}

The function sgn($x$) is the standard signum function, i.e., sgn($x)=1$ if $x$ is positive and sgn($x)=-1$ if $x$ is negative. A matching which using $i$ edges is called an $i$-matching. Let $m_{i}(\widetilde{G})$ be the number of $i$-matchings of $\widetilde{G}$. We say an edge incident to a cycle means that there is only one vertex of
this edge on the cycle.
Let $\widetilde{G}$ be a mixed graph with the characteristic polynomial of its Hermitian-adjacent matrix
$\phi(\widetilde{G}, \lambda)=\lambda^{n}+a_{1}\lambda^{n-1}+a_{2}\lambda^{n-2}+ \cdots +a_{n}.$ By Lemma
\ref{COMPON}, we only need to consider the case that $\widetilde{G}$ is connected.

\begin{lemma} \label{L32} Let $\widetilde{G}$ be a mixed unicyclic connected graph with a cycle $\widetilde{C_{q}}$.
Then for $0 \leq i \leq \lfloor \frac{n}{2}  \rfloor$, one has that
$$sgn(a_{2i})=\left\{
             \begin{array}{ll}
               (-1)^{i}, & \hbox{if $q$ is odd or $\sigma(\widetilde{C_{q}})$ is odd for $0 \leq i \leq m(\widetilde{G})$;} \\
               (-1)^{i}, & \hbox{if $q$ is even and $|\sigma(\widetilde{C_{q}})-q| \equiv 2 \ (\rm{mod} \ 4)$ for $0 \leq i \leq m(\widetilde{G})$ ;} \\
               (-1)^{i}, & \hbox{if $q$ is even and $|\sigma(\widetilde{C_{q}})-q| \equiv 0 \ (\rm{mod} \ 4)$ for $0 \leq i \leq m(\widetilde{G})$ and there exists} \\
                          & \hbox{an $i$-matching containing an edge incident to the cycle;} \\
               (-1)^{i}, & \hbox{if $q$ is even and $|\sigma(\widetilde{C_{q}})-q| \equiv 0 \ (\rm{mod} \ 4)$ for $0 \leq i \leq \frac{q}{2}$ and there exists no} \\
                          & \hbox{ $i$-matching containing an edge incident to the cycle;} \\
               0, & \hbox{otherwise.}
             \end{array}
           \right.
$$
Where $a_{2i}$ is the coefficient of $\lambda^{n-2i}$ in the characteristic polynomial
$\phi(\widetilde{G}, \lambda)$.
\end{lemma}

\begin{proof} It is obviously that $a_{0}=1$, which is equal to the stated result $(-1)^{0}$ for any value of $q$ as $m(\widetilde{G}) \geq 0$.

If $q$ is odd. Then any real elementary subgraph on an even number of vertices consists only of copies of $K_{2}$.
Such a matching of $2i$ vertices exists only if $i \leq m(\widetilde{G})$. Therefore, for $i > m(\widetilde{G})$, no real elementary subgraph
 exists, it implies that sgn$(a_{2i})=0$. Otherwise, each real elementary subgraph contributes $(-1)^{i}$ to the sum of $a_{2i}$, so we have sgn$(a_{2i})=(-1)^{i}$ for $i \leq m(\widetilde{G})$.

If $q$ is even and $i\geq 1$. Since any even cycle can be decomposed into two matchings, we again have sgn$(a_{2i})=0$ when $i > m(\widetilde{G})$. Suppose $2i < q$, then any real elementary subgraph on $2i$ vertices only contains
copies of $K_{2}$, it implies that sgn$(a_{2i})=(-1)^{i}$. Suppose $q \leq 2i \leq 2m(\widetilde{G})$. We will deal with the following two subcases according to the parity of $\sigma(\widetilde{C_{q}})$.

Subcase 1. $\sigma(\widetilde{C_{q}})$ is odd.

Then any real elementary subgraph on $2i$ vertices only contains copies of $K_{2}$, it implies that sgn$(a_{2i})=(-1)^{i}$.

Subcase 2. $\sigma(\widetilde{C_{q}})$ is even.

Then some real elementary subgraphs on $2i$ vertices contain $\widetilde{C_{q}}$ and $i-q/2$ copies of $K_{2}$ and
some real elementary subgraphs contain only $i$ copies of $K_{2}$. By Lemma \ref{L23}, one has that
\begin{eqnarray*}
a_{2i}
&=&(-1)^{2i}(m_{i-q/2}(\widetilde{G}-V(\widetilde{C_{q}})))(-1)^{\frac{1}{2}\sigma(\widetilde{C_{q}})+i-\frac{q}{2}+1}2^{1}+m_{i}(G)(-1)^{i}\\
&=&(-1)^{i}[2(m_{i-q/2}(\widetilde{G}-V(\widetilde{C_{q}})))(-1)^{\frac{|\sigma(\widetilde{C_{q}})-q|}{2}+1}+m_{i}(\widetilde{G})].
\end{eqnarray*}

Since $\sigma(\widetilde{C_{q}})$ and $q$ are even, $|\sigma(\widetilde{C_{q}})-q|$ is even.
If $|\sigma(\widetilde{C_{q}})-q| \equiv 2 \ (\rm{mod} \ 4)  $.
Then $(-1)^{\frac{|\sigma(\widetilde{C_{q}})-q|}{2}+1}=1$ and sgn$(a_{2i})=(-1)^{i}$.

If $|\sigma(\widetilde{C_{q}})-q| \equiv 0 \ (\rm{mod} \ 4)  $.
Then $a_{2i}=(-1)^{i}[m_{i}(\widetilde{G})-2(m_{i-q/2}(\widetilde{G}-V(\widetilde{C_{q}})))]$. But $m_{i}(\widetilde{G}) \geq 2(m_{i-q/2}(\widetilde{G}-V(\widetilde{C_{q}})))$ as
$2(m_{i-q/2}(\widetilde{G}-V(\widetilde{C_{q}})))$ matchings of $\widetilde{G}$ of size $i$ can be found by using the two matchings in the cycle.
Furthermore, $m_{i}(\widetilde{G}) > 2(m_{i-q/2}(\widetilde{G}-V(\widetilde{C_{q}})))$ only when there exists a matching of $\widetilde{G}$ of size $i$ that
uses an edge between $\widetilde{C_{q}}$ and $\widetilde{G}-V(\widetilde{C_{q}})$.
Then, sgn$(a_{2i})=(-1)^{i}$ if $q$ is even and $|\sigma(\widetilde{C_{q}})-q| \equiv 0 \ (\rm{mod} \ 4)$ for $0 \leq i \leq m(\widetilde{G})$ and there exists an $i$-matching containing an edge incident to the cycle.

The results follows.
\end{proof}

\begin{lemma}\label{L33} Let $\widetilde{G}$ be a mixed unicyclic connected graph with a cycle $\widetilde{C_{q}}$. Let
$$k=\left\{
             \begin{array}{ll}
               m(\widetilde{G})-1, & \hbox{if $q$ and $\sigma(\widetilde{C_{q}})$ are even and $|\sigma(\widetilde{C_{q}})-q| \equiv 0 \ (\rm{mod} \ 4)  $, and no maximum} \\
              & \hbox{matching contains an edge incident to the cycle ;} \\
               m(\widetilde{G}),  & \hbox{otherwise.}
             \end{array}
           \right.
$$
Then for $0 \leq i \leq \lfloor n/2 \rfloor$,
$$\mathrm{sgn}(a_{2i})=\left\{
             \begin{array}{ll}
               (-1)^{i}, & \hbox{if $i \leq k$;} \\
               0,  & \hbox{otherwise.}
             \end{array}
           \right.
$$

\end{lemma}

\begin{proof} If $q$ is odd or $\sigma(\widetilde{C_{q}})$ is odd, or $q$ and $\sigma(\widetilde{C_{q}})$ are even and $|\sigma(\widetilde{C_{q}})-q| \equiv 2 \ (\rm{mod} \ 4)$ for $0 \leq i \leq m(\widetilde{G})$, by
Lemma \ref{L32}, the statements follow immediately. For theremainder of the proof, we assume that $q$ and $\sigma(\widetilde{C_{q}})$ are even and $|\sigma(\widetilde{C_{q}})-q| \equiv 0 \ (\rm{mod} \ 4)  $, and $0 \leq i \leq m(\widetilde{G})$.
If there exists a maximum matching which contains an edge incident to the cycle of $\widetilde{G}$, by
Lemma \ref{L32}, the statement follows immediately.
So one can assume that there exists no maximum matching which contains an edge incident to the cycle of $\widetilde{G}$, and we just need to prove that there exists
an ($m(\widetilde{G})-1$)-matching which containing an edge incident to the cycle.

Notice that if there exists an $i$-matching which containing an edge incident to the cycle for $i > 1$, then such an $(i-1)$-matching also exists. Likewise, if there exists no $i$-matching which containing an edge incident to the cycle,
then there exists no $(i+1)$-matching which containing such an edge. Thus, there is some maximum value $k$ such that
there exists a $k$-matching containing an edge incident to the cycle, but no such $(k+1)$-matching exists.
Clearly, $k \leq m(\widetilde{G})$. Also, $k=0$ if and only if $\widetilde{G}=\widetilde{C_{q}}$.

Suppose $0 < k < m(\widetilde{G})$. That is, no maximum matching which containing an edge incident to the cycle. Consider a maximum matching $M$ and an edge $e=(u, v)$ incident to the cycle such that $u \in \widetilde{C_{q}}$ and
$v \notin \widetilde{C_{q}}$. Hence,  one has that $e \notin M$. Furthermore, since $M$ is a maximum matching and $q$ is even, every vertex on the cycle is incident to a matched edge on the cycle. Let $f$
be the matched edge in the cycle which is incident to $u$. There also must exist an edge $g \in M$ incident to $v$, otherwise the matching $M-f+e$ would be a maximum matching that contradicts $k \neq m(\widetilde{G})$. Note that
$M-f-g+e$ is a matching of size $m(\widetilde{G})-1$ which containing an edge incident to the cycle and hence $k=m(\widetilde{G})-1$.

Finally, suppose $k=0$ and hence $\widetilde{G}=\widetilde{C_{q}}$. The argument in the proof of Theorem \ref{T31} shows that $a_{2i}=(-1)^{i}$
for $2i < q$. When $2i=q$, there are exactly three real elementary subgraphs: two perfect matchings and one containing only $\widetilde{C_{q}}$. By Lemma \ref{L23}, we have $a_{q}=2(-1)^{\frac{q}{2}}+2(-1)^{\frac{\sigma(\widetilde{C_{q}})}{2}+1}=0$ as $q$ and $\sigma(\widetilde{C_{q}})$ are even and $|\sigma(\widetilde{C_{q}})-q| \equiv 0 \ (\rm{mod} \ 4)  $.

The result follows.
\end{proof}

Let $\widetilde{G}$ be a mixed graph with the characteristic polynomial of its Hermitian-adjacent matrix
$\phi(\widetilde{G}, \lambda)=\lambda^{n}+a_{1}\lambda^{n-1}+a_{2}\lambda^{n-2}+ \cdots +a_{n}.$ By Lemma
\ref{L21}, the number of positive eigenvalues of $H(\widetilde{G})$ can be determined by counting the number of sign changes of the $a_{i}'s$.
Let $k$ be the maximum value such that $a_{2k}\neq 0$. Then it does not matter what the values of $a_{2i+1}$ are for $0 \leq i < k$, as they
do not affect the number of sign changes since the sign of $a_{2i+1}$ is either the same as $a_{2i}$ or $a_{2i+2}$, or $a_{2i+1}=0$. Thus, we only consider the odd coefficients $a_{2i+1}$ when $i \geq k$. We
characterize those now.

\begin{lemma}\label{L34}
Let $\widetilde{G}$ be a  mixed unicyclic graph with the cycle $\widetilde{C_{q}}$. Let $k$ be the maximum value such that $a_{2k}\neq 0$. Then $a_{2i+1}=0$ for all $i > k$.
\end{lemma}
\begin{proof}
This is trivial in the case where $q$ is even, because there exists no real elementary subgraph on an odd
number of vertices. Assume $q$ is odd. If there exists a real elementary subgraph on $2i+1$ vertices, then there
exists a real elementary subgraph on $2i$ vertices, which is found by replacing $\widetilde{C_{q}}$ (which must be included)
in the real elementary subgraph with $(q-1)/2$ copies of $K_{2}$ to get a matching of size $i$ which implies
that $a_{2i} \neq 0$ for $i > k$. Thus, $a_{2i+1}=0$ for $i > k$.
\end{proof}

Our research for the number of sign changes has now been reduced to finding the sign of $a_{2k+1}$.
All the odd coefficients are 0 when $q$ is even or $\sigma(\widetilde{C_{q}})$ is odd. The only case needed to consider is that $q$ is odd and $\sigma(\widetilde{C_{q}})$ is even, and note that
$k$ has been defined in this case which equal to $m(\widetilde{G})$. A real elementary subgraph on $2m(\widetilde{G})+1$ vertices must contain
$\widetilde{C_{q}}$ and $(2m(\widetilde{G})+1-q)/2$ copies of $K_{2}$ from $\widetilde{G}-V(\widetilde{C_{q}})$. Therefore, we have
\begin{eqnarray*}
a_{2m(\widetilde{G})+1}
&=&2m_{m(\widetilde{G})-\frac{q-1}{2}}(\widetilde{G}-V(\widetilde{C_{q}}))(-1)^{\frac{1}{2}\sigma(\widetilde{C_{q}})+\frac{2m(\widetilde{G})-q+1}{2}+1}\\
&=&2m_{m(\widetilde{G})-\frac{q-1}{2}}(\widetilde{G}-V(\widetilde{C_{q}}))(-1)^{\frac{\sigma(\widetilde{C_{q}})+2m(\widetilde{G})-q+3}{2}}\\
&=&(-1)^{m(\widetilde{G})}2m_{m(\widetilde{G})-\frac{q-1}{2}}(\widetilde{G}-V(\widetilde{C_{q}}))(-1)^{\frac{|\sigma(\widetilde{C_{q}})-q|+1}{2}}.\\
\end{eqnarray*}
So,
$$\mathrm{sgn}(a_{2m(\widetilde{G})+1})=\left\{
             \begin{array}{ll}
               (-1)^{m(\widetilde{G})+\frac{|\sigma(\widetilde{C_{q}})-q|+1}{2}}, & \hbox{if $q$ is odd and $\sigma(\widetilde{C_{q}})$ is even} \\
               & \hbox{ and $m_{m(\widetilde{G})-\frac{q-1}{2}}(\widetilde{G}-V(\widetilde{C_{q}})) > 0$;} \\
               0,  & \hbox{otherwise.}
             \end{array}
           \right.
$$
Therefore, sgn$(a_{2m(\widetilde{G})+1})$ differents from sgn$(a_{2m(\widetilde{G})})=(-1)^{m(\widetilde{G})}$ when $q$ is odd,
$\sigma(\widetilde{C_{q}})$ is even, $m_{m(\widetilde{G})-\frac{q-1}{2}}(\widetilde{G}-V(\widetilde{C_{q}})) > 0$ and $|\sigma(\widetilde{C_{q}})-q| \equiv 1 \ (\rm{mod} \ 4)  $.
And sgn$(a_{2m(\widetilde{G})+1})$ and sgn$(a_{2m(\widetilde{G})})$ are same if $q$ is odd,
$\sigma(\widetilde{C_{q}})$ is even, $m_{m(\widetilde{G})-\frac{q-1}{2}}(\widetilde{G}-V(\widetilde{C_{q}})) > 0$ and $|\sigma(\widetilde{C_{q}})-q| \equiv 3 \ (\rm{mod} \ 4)  $.
The requirement $m_{m(\widetilde{G})-\frac{q-1}{2}}(\widetilde{G}-V(\widetilde{C_{q}})) > 0$ means that there exists a maximum matching of $\widetilde{G}$ that does not use any edge between $\widetilde{C_{q}}$ and $\widetilde{G}-V(\widetilde{C_{q}})$. This is equivalent to $2m(\widetilde{G})+1=2m(\widetilde{G}-V(\widetilde{C_{q}}))+q$.

The number of positive eigenvalues can now be determined by counting the number of sign changes of the $a_{i}'s$. The number of zero eigenvalues is $n-i$, where $i$ is the largest value such that $a_{i}\neq 0$.
Hence the inertia index of a mixed unicyclic graph can then be computed by
considering the size of a maximum matching of $\widetilde{G}$ and, if necessary, $\widetilde{G}-V(\widetilde{C_{q}})$. From the above analysis,
 the following Theorem \ref{UNI} is proved.

\begin{theorem}\label{UNI}
Let $\widetilde{G}$ be a  mixed unicyclic graph with the cycle $\widetilde{C_{q}}$. Then we have
$$  (p^{+}(\widetilde{G}), n^{-}(\widetilde{G})) = \left\{
  \begin{array}{ll}
      (m(\widetilde{G})-1, m(\widetilde{G})-1), & \hbox{ if $q$ and $\sigma(\widetilde{C_{q}})$ are even, $|\sigma(\widetilde{C_{q}})-q| \equiv 0 \ (\rm{mod} \ 4)  $ } \\
              & \hbox{ and no maximum matching contains an edge } \\
              & \hbox{ incident to the cycle;} \\
      (m(\widetilde{G})+1, m(\widetilde{G})), & \hbox{ if $q$ is odd, $\sigma(\widetilde{C_{q}})$ is even, $|\sigma(\widetilde{C_{q}})-q| \equiv 1 \ (\rm{mod} \ 4)  $ } \\
                 & \hbox{ and $m(\widetilde{G})=m(\widetilde{G}-V(\widetilde{C_{q}}))+\frac{q-1}{2}$ ;} \\
      (m(\widetilde{G}), m(\widetilde{G})+1), & \hbox{ if $q$ is odd, $\sigma(\widetilde{C_{q}})$ is even, $|\sigma(\widetilde{C_{q}})-q| \equiv 3 \ (\rm{mod} \ 4)  $ } \\
                      & \hbox{ and $m(\widetilde{G})=m(\widetilde{G}-V(\widetilde{C_{q}}))+\frac{q-1}{2}$ ;} \\
      (m(\widetilde{G}), m(\widetilde{G})), & \hbox{otherwise.}
    \end{array}
  \right.
$$
\end{theorem}

\begin{theorem} \label{T36}
Let $\widetilde{G}$ be a  mixed graph. Then
$$m(\widetilde{G})-c(\widetilde{G}) \leq p^{+}(\widetilde{G}) \leq m(\widetilde{G})+c(\widetilde{G}), m(\widetilde{G})-c(\widetilde{G}) \leq n^{-}(\widetilde{G}) \leq m(\widetilde{G})+c(\widetilde{G}). $$
\end{theorem}
\begin{proof} The proof is by induction on $c(\widetilde{G})$. If $c(\widetilde{G})=0$, then $\widetilde{G}$ is a tree. By Lemma \ref{TREE}, we have
$p^{+}(\widetilde{G})=n^{-}(\widetilde{G})=m(\widetilde{G})$. If $c(\widetilde{G})=1$, then $\widetilde{G}$ is a unicyclic graph. By Theorem \ref{UNI}, the result holds immediately.
In the following, suppose $c(\widetilde{G}) \geq 2$. Let $v$ be a vertex lying on a cycle of $\widetilde{G}$ and denote $\widetilde{H}=\widetilde{G}-v$.
Thus, $c(\widetilde{H}) \leq c(\widetilde{G})-1$. Applying the induction to $\widetilde{H}$, we have
$$m(\widetilde{H})-c(\widetilde{H}) \leq p^{+}(\widetilde{H}) \leq m(\widetilde{H})+c(\widetilde{H}).$$

By Lemma \ref{DV}, we have
$$p^{+}(\widetilde{G}) \leq p^{+}(\widetilde{H})+1\leq m(\widetilde{H})+c(\widetilde{H})+1\leq m(\widetilde{H})+c(\widetilde{G})\leq m(\widetilde{G})+c(\widetilde{G})$$
and
$$p^{+}(\widetilde{G})\geq p^{+}(\widetilde{H}) \geq m(\widetilde{H})-c(\widetilde{H}) \geq (m(\widetilde{G})-1)-(c(\widetilde{G})-1)=m(\widetilde{G})-c(\widetilde{G}).$$

The discussion for $n^{-}(\widetilde{G})$ is similar and is omitted.
\end{proof}

\begin{corollary}\label{C37}
Let $\widetilde{G}$ be a mixed graph which contains at least one cycle. If $p^{+}(\widetilde{G})= m(\widetilde{G})+c(\widetilde{G})$, then for any vertex $v$
lying on a cycle of $\widetilde{G}$, one has

{\em(i)} $p^{+}(\widetilde{G}-v)=p^{+}(\widetilde{G})-1$;

{\em(ii)} $p^{+}(\widetilde{G}-v)=m(\widetilde{G}-v)+c(\widetilde{G}-v)$;

{\em(iii)} $m(\widetilde{G}-v)=m(\widetilde{G})$;

{\em(iv)} $c(\widetilde{G}-v)=c(\widetilde{G})-1$;

{\em(v)}  $rank(\widetilde{G})-2 \leq rank(\widetilde{G}-v) \leq rank(\widetilde{G})-1$;

{\em(vi)} $v$ is not a quasi-pendant vertex;

{\em(vii)} Any two cycles of $\widetilde{G}$ are vertex disjoint.

\end{corollary}

\begin{corollary}\label{C38}
Let $G$ be a mixed graph which contains at least one cycle. If $p^{+}(\widetilde{G})= m(\widetilde{G})-c(\widetilde{G})$, then for any vertex $v$
lying on a cycle of $G$, one has

{\em(i)} $p^{+}(\widetilde{G}-v)=p^{+}(\widetilde{G})$;

{\em(ii)} $p^{+}(\widetilde{G}-v)=m(\widetilde{G}-v)+c(\widetilde{G}-v)$;

{\em(iii)} $m(\widetilde{G}-v)=m(\widetilde{G})-1$;

{\em(iv)} $c(\widetilde{G}-v)=c(\widetilde{G})-1$;

{\em(v)}  $rank(\widetilde{G})-1 \leq rank(\widetilde{G}-v) \leq rank(\widetilde{G})$;

{\em(vi)} $v$ is not a quasi-pendant vertex;

{\em(vii)} Any two cycles of $\widetilde{G}$ are vertex disjoint.

\end{corollary}

By Theorem \ref{UNI} and the similar methods in \cite{FANYZ} which used for discussing the positive and negative inertia index of undirected graph, we can get the following theorems. Here we only give the proofs of Theorems \ref{T39} and \ref{T312} in the appendix.

\begin{theorem}\label{T39}
Let $\widetilde{G}$ be a mixed connected graph. Then $p^{+}(\widetilde{G})= m(\widetilde{G})+c(\widetilde{G})$ if and only if the following three conditions all holds.

{\em(i)} Any two cycles of $\widetilde{G}$ share no common vertices;

{\em(ii)} For each cycle $\widetilde{C_{q}}$ of $\widetilde{G}$, $q$ is odd, $\sigma(\widetilde{C_{q}})$ is even and $|\sigma(\widetilde{C_{q}})-q| \equiv 1 \ (\rm{mod} \ 4)  $;

{\em(iii)} $m(T_{\widetilde{G}})=m([T_{\widetilde{G}}])$.
\end{theorem}

If $\widetilde{G} \in \mathscr{G} $ and $\widetilde{G}$ contains only odd cycles, it can be checked that $m(T_{\widetilde{G}})=m([T_{\widetilde{G}}])$ if and only if there exists a maximum matching $M(\widetilde{G})$ of $\widetilde{G}$ such that $M(\widetilde{G}) \cap \mathscr{F}(\widetilde{G}) =\emptyset$. Then we have an alternative version of Theorem \ref{T39} in the following.

\begin{theorem}
Let $\widetilde{G}$ be a mixed connected graph. Then $p^{+}(\widetilde{G})= m(\widetilde{G})+c(\widetilde{G})$ if and only if the following three conditions all holds.

{\em(i)} Any two cycles of $\widetilde{G}$ share no common vertices;

{\em(ii)} For each cycle $\widetilde{C_{q}}$ of $\widetilde{G}$, $q$ is odd, $\sigma(\widetilde{C_{q}})$ is even and $|\sigma(\widetilde{C_{q}})-q| \equiv 1 \ (\rm{mod} \ 4)  $;

{\em(iii)} There exists a maximum matching $M(\widetilde{G})$ of $\widetilde{G}$ such that $M(\widetilde{G}) \cap \mathscr{F}(\widetilde{G}) =\emptyset$.
\end{theorem}

\begin{theorem}
Let $\widetilde{G}$ be a mixed connected graph. Then $n^{-}(\widetilde{G})= m(\widetilde{G})+c(\widetilde{G})$ if and only if $\widetilde{G}$ satisfies both of the first two conditions and either one of the last two conditions;

{\em(i)} Any two cycles of $\widetilde{G}$ share no common vertices;

{\em(ii)} For each cycle $\widetilde{C_{q}}$ of $\widetilde{G}$, $q$ is odd, $\sigma(\widetilde{C_{q}})$ is even and $|\sigma(\widetilde{C_{q}})-q| \equiv 3 \ (\rm{mod} \ 4)  $;

{\em(iii)} There exists a maximum matching $M(\widetilde{G})$ of $\widetilde{G}$ such that $M(\widetilde{G}) \cap \mathscr{F}(\widetilde{G}) =\emptyset$;

{\em(iv)} $m(T_{\widetilde{G}})=m([T_{\widetilde{G}}])$.
\end{theorem}

\begin{theorem}\label{T312}
Let $\widetilde{G}$ be a mixed connected graph. Then $p^{+}(\widetilde{G})= m(\widetilde{G})-c(\widetilde{G})$  (or $n^{-}(\widetilde{G})= m(\widetilde{G})-c(\widetilde{G})$) if and only if the following three conditions all hold.

{\em(i)} Any two cycles of $\widetilde{G}$ share no common vertices;

{\em(ii)} For each cycle $\widetilde{C_{q}}$ of $\widetilde{G}$, $q$ is even, $\sigma(\widetilde{C_{q}})$ is even and $|\sigma(\widetilde{C_{q}})-q| \equiv 0 \ (\rm{mod} \ 4)  $;

{\em(iii)} $m(T_{\widetilde{G}})=m([T_{\widetilde{G}}])$.
\end{theorem}

\section*{Acknowledgments}

This work was supported by the National Natural Science Foundation of China
(No. 11731002), the Fundamental Research Funds for the Central Universities (No. 2016JBM071) and the $111$ Project of China.

\section*{Appendix: Proof of Theorem \ref{T39} }

\begin{lemma}\label{A1}{\rm\cite{FANYZ}}
Let $G \in \mathscr{G}$. If  $m(T_{G})=m([T_{G}])$, then $T_{G}$  contains a non-cyclic pendants vertex.
If $v$ is the vertex in $T_{G}$ adjacent to such pendant vertex, then $v$ is also non-cyclic. In other words, $G$ contains at least one pendant vertex, and any quasi-pendant
vertex of $G$ lies outside of cycles.
\end{lemma}

\begin{lemma}\label{A2}{\rm\cite{RULA}}
Let $G$ be a graph with at least one cycle. Suppose that all cycles of $G$
are pairwise-disjoint and each cycle is odd, then $m(T_{G})=m([T_{G}])$ if and only if
$m(G)=\sum_{C \in \mathscr{L}(G)}m(C)+m([T_{G}])$, where $\mathscr{L}(G)$ denotes the set of all cycles in $G$.
\end{lemma}

\begin{lemma}\label{A3}{\rm\cite{GONG}}
Let $G$ be a graph with a quasi-pendant vertex $v$. Then $m(G-v)=m(G)-1$.
\end{lemma}

\begin{lemma}\label{A6}{\rm\cite{FANYZ}}
Let $G \in \mathscr{G}$. If there exists a maximum matching $M(G)$ of $G$ such that
$M(G) \cap \mathscr{F}(G) =\emptyset$, then $m(G)=\sum_{C \in \mathscr{L}(G)}m(C)+m([T_{G}])$, where $\mathscr{L}(G)$ denotes the set of all cycles in $G$. If addition, each cycles of $G$ has odd length, then $m(T_{G})=m([T_{G}])$.
\end{lemma}

\begin{lemma}\label{A4}{\rm\cite{FANYZ}}
Let $K$ be a graph such that any two cycles share no common vertices. Let
$G$ be a graph obtained from $K$ and a cycle $C_{s}$ (disjoint to $K$) by adding an edge between
a vertex $x$ of $C_{s}$ and a vertex $y$ of $K$. If $p(G)=m(G)-c(G)$, then

(i) $s$ is a multiple of 4;

(ii) the edge $xy$ does not belong to any maximum matching of $G$;

(iii) each maximum matching of $K$ covers $y$;

(iv) $m(K+x)=m(K)$;

(v) $m(G)=m(C_{s})+m(K)=m(C_{s})+m(K+x)$.
\end{lemma}

\begin{lemma}\label{A5}{\rm\cite{FANYZ}}
Let $G \in \mathscr{G}$. If  $p(G)=m(G)-c(G)$, then for any maximum matching $M(G)$ of $G$, one has $M(G) \cap \mathscr{F}(\widetilde{G}) =\emptyset$.
\end{lemma}

$\mathbf{Proof}$ $\mathbf{of}$ $\mathbf{Theorem~\ref{T39}:}$

\begin{proof}(Sufficiency.) We use induction on the order of $\widetilde{G}$. If $\widetilde{G}$ is a disjoint union of trees and/or cycles satisfies the condition {\em(ii)}, clearly the result holds by Lemma \ref{TREE} and Theorem \ref{T31}. So we assume
$\widetilde{G} \in \mathscr{G}$. As $m(T_{\widetilde{G}})=m([T_{\widetilde{G}}])$, by Lemma \ref{A1}, $\widetilde{G}$ contains a pendant vertex $u$ and a quasi-pendant vertex $v$ adjacent to $u$, and $v$ lies outside any cycle of $\widetilde{G}$. Let $\widetilde{H}=\widetilde{G}-\{u, v\}$.
By Lemma \ref{A3}, $m(T_{\widetilde{H}})=m([T_{\widetilde{H}}])$ and $\widetilde{H}$ satisfies the three conditions (i)-(iii) of this theorem. By induction we have $p^{+}(\widetilde{H})= m(\widetilde{H})+c(\widetilde{H})$. So by Lemmas \ref{L29} and \ref{A3}, one has that
$$p^{+}(\widetilde{G})=p^{+}(\widetilde{H})+1= m(\widetilde{H})+c(\widetilde{H})+1= m(\widetilde{G})+c(\widetilde{G}).$$

(Necessity.) Let $\widetilde{G}$ be a mixed graph such that $p^{+}(\widetilde{G})= m(\widetilde{G})+c(\widetilde{G})$. If $\widetilde{G}$ is a forest then $\widetilde{G}$ clearly
satisfies the three conditions (i)-(iii). Assume that $\widetilde{G}$ has at least one cycle. The assertion (i) follows from Corollary \ref{C37}.

We assert that for each cycle $\widetilde{C_{q}}$ of $\widetilde{G}$, we have $q$ is odd, $\eta(\widetilde{C_{q}})$ is even and $|\eta(\widetilde{C_{q}})-q| \equiv 1 \ (\rm{mod} \ 4)  $. If $c(\widetilde{G})=1$, the result holds by Theorem \ref{UNI}. Now assume
$c(\widetilde{G})=l \geq 2$. If there exist a cycle, say $\widetilde{C^{'}}$, which  not satisfies the condition (ii), then by
deleting an arbitrary vertex of each cycle of $\widetilde{G}$ except $\widetilde{C'}$, we get a graph $\widetilde{H}$ with $c(\widetilde{H})=1$ and
$p^{+}(\widetilde{H}) \leq m(\widetilde{H})$ by Theorem \ref{UNI}. By Theorem \ref{T36}, we have
$$p^{+}(\widetilde{G}) \leq p^{+}(\widetilde{H})+l-1 \leq m(\widetilde{H})+l-1< m(\widetilde{G})+c(\widetilde{G}),$$
a contradiction.

We prove the assertion (iii) by the induction on the order of $\widetilde{G}$. If $\widetilde{G}$ is a disjoint union of cycles,
the result follows. So we assume that $\widetilde{G} \in \mathscr{G}$. First suppose that $\widetilde{G}$ contains a pendant
vertex, say $x$, and a quasi-pendant vertex $y$ that is adjacent to $x$. By Corollary \ref{C37}, $y$ is not lying on
any cycle. Let $\widetilde{H}=\widetilde{G}-\{x, y\}$. Then by Lemmas \ref{L29} and \ref{A3}, we have
$$p^{+}(\widetilde{H})=p^{+}(\widetilde{G})-1=m(\widetilde{G})+c(\widetilde{G})-1=m(\widetilde{H})+c(\widetilde{H}).$$
By induction we have $m(T_{\widetilde{H}})=m([T_{\widetilde{H}}])$, hence by Lemma \ref{A3}, one has that
$$m(T_{\widetilde{G}})= m(T_{\widetilde{H}})+1=m([T_{\widetilde{G}}])+1=m([T_{\widetilde{G}}]).$$

Now suppose that $\widetilde{G}$ contains no pendant vertices. Then $\widetilde{G}$ contains a pendant
cycle, say $\widetilde{C}$ such that $\widetilde{C}$ has exactly one vertex say $u$ that is adjacent to a vertex $v$ outside $\widetilde{C}$.
Let $\widetilde{K}=\widetilde{G}-\widetilde{C}$. By Corollary \ref{C37}, we have $p^{+}(\widetilde{G}-u)=m(\widetilde{G}-u)+c(\widetilde{G}-u)$ and $m(\widetilde{G}-u)=m(\widetilde{G})$, which implies that
$p^{+}(\widetilde{K})=m(\widetilde{K})+c(\widetilde{K})$ from the former equality ($p^{+}(\widetilde{G}-u)=m(\widetilde{G}-u)+c(\widetilde{G}-u)$), and $m(\widetilde{G})=m(\widetilde{C})+m(\widetilde{K})$ from the
latter equality ($m(\widetilde{G}-u)=m(\widetilde{G})$). By the induction, we have $m(T_{\widetilde{K}})=m([T_{\widetilde{K}}])$. So $\widetilde{K}$ has a maximum
matching $M(\widetilde{K})$ such that $M(\widetilde{K}) \cap \mathscr{F}(\widetilde{K}) = \emptyset$ by Lemma \ref{A2}. Let $M(\widetilde{C})$ be a maximum matching
of $\widetilde{C}$. Then, $M(\widetilde{G})=M(\widetilde{K}) \cup M(\widetilde{C})$ is a maximum matching of $\widetilde{G}$, which satisfies $M(\widetilde{G}) \cap \mathscr{F}(\widetilde{G})= \emptyset$.
Again by Lemma \ref{A2}, we get $m(T_{\widetilde{G}})=m([T_{\widetilde{G}}])$.
\end{proof}

$\mathbf{Proof}$ $\mathbf{of}$ $\mathbf{Theorem~\ref{T312}:}$

\begin{proof}
(Sufficiency.) We use induction on the order of $\widetilde{G}$. If $\widetilde{G}$ is a disjoint union of trees and/or cycles satisfies the condition {\em(ii)}, clearly the result holds by Lemma \ref{TREE} and Theorem \ref{T31}. So we assume
$\widetilde{G} \in \mathscr{G}$. As $m(T_{\widetilde{G}})=m([T_{\widetilde{G}}])$, by Lemma \ref{A1}, $\widetilde{G}$ contains a pendant vertex $x$ and a quasi-pendant vertex $y$ adjacent to $x$, and $y$ lies outside any cycle of $\widetilde{G}$. Let $\widetilde{H}=\widetilde{G}-\{x, y\}$.
By Lemma \ref{A3}, $m(T_{\widetilde{H}})=m([T_{\widetilde{H}}])$ and $\widetilde{H}$ satisfies the three conditions (i)-(iii). By induction we have $p^{+}(\widetilde{H})= m(\widetilde{H})-c(\widetilde{H})$. So by Lemmas \ref{L29} and \ref{A3}, one has that
$$p^{+}(\widetilde{G})=p^{+}(\widetilde{H})+1= m(\widetilde{H})-c(\widetilde{H})+1= m(\widetilde{G})-c(\widetilde{G}).$$

(Necessity.) Let $\widetilde{G}$ be a mixed graph such that $p^{+}(\widetilde{G})= m(\widetilde{G})-c(\widetilde{G})$. The proof for (i) and
(ii) goes parallel as in Theorem \ref{T39}, thus omitted. We now prove $m(T_{\widetilde{G}})=m([T_{\widetilde{G}}])$ by
the induction on the order of $\widetilde{G}$. First assume that $\widetilde{G}$ contains a pendant vertex $x$ that
is adjacent to a vertex $y$. Then $y$ is lying outside any cycle of $\widetilde{G}$ by Corollary \ref{C38}.
Let $\widetilde{H}=\widetilde{G}-\{x, y\}$. By Lemmas \ref{A3} and \ref{L29}, $p^{+}(\widetilde{H})= m(\widetilde{H})-c(\widetilde{H})$. So, by induction
$m(T_{\widetilde{H}})=m([T_{\widetilde{H}}])$, and hence $m(T_{\widetilde{G}})=m(T_{\widetilde{H}})+1=m([T_{\widetilde{H}}])+1=m([T_{\widetilde{G}}])$ by Lemma \ref{A3}.

If $\widetilde{G}$ contains no pendant vertices, then there exists a pendant cycle $\widetilde{C}$ of $\widetilde{G}$, which contains exactly one vertex, says $u$ that is adjacent to a vertex $v$ outside $\widetilde{C}$. Let $\widetilde{K}=\widetilde{G}-\widetilde{C}$,
and let $\widetilde{H}=\widetilde{K}+u$. Let $w$ be a vertex of $\widetilde{C}$ adjacent to $u$. By Corollary \ref{C38},
$p^{+}(\widetilde{G}-w)= m(\widetilde{G}-w)-c(\widetilde{G}-w)$. Repeatedly deleting the pendant vertex and the quasi-pendant
vertices of $\widetilde{C}-w$ until we arrive at the graph $\widetilde{H}$, we get $p^{+}(\widetilde{H})= m(\widetilde{H})-c(\widetilde{H})$ by Lemmas \ref{A3} and \ref{L29}. By the induction, $m(T_{\widetilde{H}})=m([T_{\widetilde{H}}])$. Suppose that $\widetilde{C}=\widetilde{C_{1}}, \widetilde{C_{2}}, \cdots, \widetilde{C_{l}}$
are all cycles contained in $\widetilde{G}$. By Lemmas \ref{A5} and \ref{A6},
$$m(\widetilde{G})=m([T_{\widetilde{G}}])+\frac{\sum_{i=1}^{l}|V(\widetilde{C_{i}})|}{2}.$$
By a similar discussion, we also have
$$m(\widetilde{H})=m([T_{\widetilde{H}}])+\frac{\sum_{i=2}^{l}|V(\widetilde{C_{i}})|}{2}.$$
Obviously, $T_{\widetilde{G}}$ is isomorphic to $T_{\widetilde{H}}$. Thus $m(T_{\widetilde{G}})=m(T_{\widetilde{H}})$.
Noting that $m(\widetilde{H})=m(\widetilde{K})$ and $m(\widetilde{G})=m(\widetilde{C_{1}})+m(\widetilde{K})$ by Lemma \ref{A4}, we finally have
\begin{eqnarray*}
m(T_{\widetilde{G}})
&=&m(T_{\widetilde{H}})=m([T_{\widetilde{H}}])=m(\widetilde{H})-\frac{\sum_{i=2}^{l}|V(\widetilde{C_{i}})|}{2}\\
&=&m(\widetilde{K})-\frac{\sum_{i=2}^{l}|V(\widetilde{C_{i}})|}{2}=(m(\widetilde{G})-m(\widetilde{C_{1}}))-\frac{\sum_{i=2}^{l}|V(\widetilde{C_{i}})|}{2}\\
&=&m(\widetilde{G})-\frac{\sum_{i=1}^{l}|V(\widetilde{C_{i}})|}{2}\\
&=&m([T_{\widetilde{G}}]).\\
\end{eqnarray*}
\end{proof}

\end{document}